\documentclass[11pt]{amsart}
\usepackage{amssymb}
\usepackage{mathrsfs}
\usepackage{enumerate}
\makeatletter
\@namedef{subjclassname@2010}{%
  \textup{2010} Mathematics Subject Classification}
\makeatother
\newtheorem{lemma}{Lemma}
\newtheorem{theorem}{Theorem}[section]
\newtheorem{proposition}{Proposition}
\newtheorem{corollary}{Corollary}
\newtheorem{definition}{Definition}

\newtheorem{example}{Example}

\begin{document}

\title{GENERALIZED FUSION FRAMES IN HILBERT SPACES}

\author[V. Sadri]{Vahid Sadri}
\address{Department of Mathematics, Faculty of Tabriz  Branch,\\ Technical and Vocational University (TUV), East Azarbaijan
, Iran}
\email{vahidsadri57@gmail.com}

\author[Gh. Rahimlou]{GHOLAMREZA RAHIMLOU}
\address{Department of Mathematics, Faculty of Tabriz  Branch,\\ Technical and Vocational University (TUV), East Azarbaijan
, Iran}
\email{grahimlou@gmail.com}

\author[R. Ahmadi]{Reza Ahmadi}
\address{Institute of Fundamental Sciences\\University of Tabriz\\, Iran\\}
\email{rahmadi@tabrizu.ac.ir}

\author[R. Zarghami Farfar]{Ramazan Zarghami Farfar}
\address{Dapartement of Geomatic and Mathematical\\Marand Faculty of Technical and Engineering\\University of Tabriz\\, Iran\\}
\email{zarghamir@gmail.com}


\begin{abstract}
After introducing g-frames and fusion frames by Sun  and Casazza, combining these frames together is an interesting topic for research. In this paper, we introduce the generalized fusion frames or g-fusion frames for Hilbert spaces and give characterizations of these frames from the viewpoint of closed range and g-fusion frame sequences. Also, the canonical dual g-fusion frames are presented and we  introduce  Parseval g-fusion frames.
\end{abstract}



\maketitle

\section{Introduction}
During the past few years, the theory of frames have been growing
rapidly and  new topics about them are discovered almost every year.
For example, generalized frames (or g-frames), subspaces of frames
(or fusion frames), continuous frames (or c-frames), $k$-frames,
controlled frames and  the combination of each two of them, lead
 to c-fusion frames, g-c-frames, c-g-frames, c$k$-frames,
c$k$-fusion frames and etc. The purpose of this paper is to
introduce and review some of the generalized fusion frames (or g-fusion
frames) and their operators. Then, we will get some useful
propositions about these frames and finally, we will study  g-fusion
frame sequences.

Throughout this paper, $H$ and $K$ are separable Hilbert spaces and $\mathcal{B}(H,K)$ is the collection of all the bounded linear operators of $H$ into $K$. If $K=H$, then $\mathcal{B}(H,H)$ will be denoted by $\mathcal{B}(H)$. Also, $\pi_{V}$ is the orthogonal projection from $H$ onto a closed subspace $V\subset H$ and  $\lbrace H_j\rbrace_{j\in\Bbb J}$ is a sequence of Hilbert spaces where $\Bbb J$ is a subset of $\Bbb Z$. It is easy to check that if $u\in\mathcal{B}(H)$ is an invertible operator, then (\cite{ga})
$$\pi_{uV}u\pi_{V}=u\pi_{V}.$$
\begin{definition}\textbf{(frame)}.
Let $\{f_j\}_{j\in\Bbb J}$ be a sequence of members of $H$. We say that $\{f_j\}_{j\in\Bbb J}$ is a  frame  for $H$ if there exists $0<A\leq B<\infty$ such that for each $f\in H$
\begin{eqnarray*}
A\Vert f\Vert^2\leq\sum_{j\in\Bbb J}\vert\langle f,f_j\rangle\vert^2\leq B\Vert f\Vert^2.
\end{eqnarray*}
\end{definition}
\begin{definition}\textbf{(g-frame)}
A family $\lbrace \Lambda_j\in\mathcal{B}(H,H_j)\rbrace_{j\in\Bbb J}$ is called a g-frame for $H$ with respect to $\lbrace H_j\rbrace_{j\in\Bbb J}$,  if there exist $0<A\leq B<\infty$ such that
\begin{equation} \label{1}
  A\Vert f\Vert^2\leq\sum_{j\in\Bbb J}\Vert \Lambda_{j}f\Vert^2\leq B\Vert f\Vert^2, \ \  f\in H.
\end{equation}
\end{definition}
\begin{definition}\textbf{(fusion frame)}.
Let $\{W_j\}_{j\in\Bbb J}$ be a family of closed subspaces of $H$ and $\{v_j\}_{j\in\Bbb J}$ be a family of weights (i.e. $v_j>0$ for any $j\in\Bbb J$). We say that $(W_j, v_j)$ is a fusion frame  for $H$ if there exists $0<A\leq B<\infty$ such that for each $f\in H$
\begin{eqnarray*}
A\Vert f\Vert^2\leq\sum_{j\in\Bbb J}v^{2}_j\Vert \pi_{W_j}f\Vert^2\leq B\Vert f\Vert^2.
\end{eqnarray*}
\end{definition}
If an operator $ u$ has closed range, then there exists a right-inverse operator $u^ \dagger$ (pseudo-inverse of $u$) in the following sences (see \cite{ch}).

\begin{lemma}\label{l1}
Let $u\in\mathcal{B}(K,H)$  be a bounded operator with closed range $\mathcal{R}_{u}$. Then there exists a bounded operator $u^\dagger \in\mathcal{B}(H,K)$ for which
$$uu^{\dagger} x=x, \ \ x\in \mathcal{R}_{u}.$$
\end{lemma}
\begin{lemma}\label{Ru}
Let $u\in\mathcal{B}(K,H)$. Then the following assertions holds:
\begin{enumerate} \item
 $\mathcal{R}_u$ is closed in $H$ if and only if $\mathcal{R}_{u^{\ast}}$ is closed in $K$.
\item $(u^{\ast})^\dagger=(u^\dagger)^\ast$.
\item
The orthogonal projection of $H$ onto $\mathcal{R}_{u}$ is given by $uu^{\dagger}$.
\item
The orthogonal projection of $K$ onto $\mathcal{R}_{u^{\dagger}}$ is given by $u^{\dagger}u$.\item$\mathcal{N}_{{u}^{\dagger}}=\mathcal{R}^{\bot}_{u}$ and $\mathcal{R}_{u^{\dagger}}=\mathcal{N}^{\bot}_{u}$.
 \end{enumerate}
\end{lemma}
\section{Generalized Fusion Frames and Their Operators}
 We define the space $\mathscr{H}_2:=(\sum_{j\in\Bbb J}\oplus H_j)_{\ell_2}$ by
\begin{eqnarray}
\mathscr{H}_2=\big\lbrace \lbrace f_j\rbrace_{j\in\Bbb J} \ : \ f_j\in H_j , \ \sum_{j\in\Bbb J}\Vert f_j\Vert^2<\infty\big\rbrace.
\end{eqnarray}
with the inner product defined by
$$\langle \lbrace f_j\rbrace, \lbrace g_j\rbrace\rangle=\sum_{j\in\Bbb J}\langle f_j, g_j\rangle.$$
It is clear that $\mathscr{H}_2$ is a Hilbert space with pointwise operations.
\begin{definition}
Let $W=\lbrace W_j\rbrace_{j\in\Bbb J}$ be a family of closed subspaces of $H$, $\lbrace v_j\rbrace_{j\in\Bbb J}$ be a family of weights, i.e. $v_j>0$  and $\Lambda_j\in\mathcal{B}(H,H_j)$ for each $j\in\Bbb J$. We say $\Lambda:=(W_j, \Lambda_j, v_j)$ is a \textit{generalized fusion frame} (or \textit{g-fusion frame} ) for $H$ if there exists $0<A\leq B<\infty$ such that for each $f\in H$
\begin{eqnarray}\label{g}
A\Vert f\Vert^2\leq\sum_{j\in\Bbb J}v_j^2\Vert \Lambda_j \pi_{W_j}f\Vert^2\leq B\Vert f\Vert^2.
\end{eqnarray}
\end{definition}
We call $\Lambda$ a \textit{Parseval g-fusion frame}  if $A=B=1$. When the right hand of (\ref{g}) holds, $\Lambda$ is called a \textit{g-fusion Bessel sequence}  for $H$ with bound $B$. If $H_j=H$ for all $j\in\Bbb J$ and $\Lambda_j=I_H$, then we get the fusion frame $(W_j, v_j)$ for $H$. Throughout this paper, $\Lambda$ will be a triple $(W_j, \Lambda_j, v_j)$ with $j\in\Bbb J$ unless otherwise stated.
\begin{proposition}\label{2.2}
Let $\Lambda$ be a g-fusion Bessel sequence for $H$ with bound $B$. Then for each sequence $\lbrace f_j\rbrace_{j\in\Bbb J}\in\mathscr{H}_2$, the series $\sum_{j\in\Bbb J}v_j \pi_{W_j}\Lambda_{j}^{*}f_j$ converges unconditionally.
\end{proposition}
\begin{proof}
Let $\Bbb I$ be a finite subset of $\Bbb J$, then
\begin{align*}
\Vert \sum_{j\in\Bbb I}v_j\pi_{W_j}\Lambda_j^{*} f_j\Vert&=\sup_{\Vert g\Vert=1}\big\vert\langle\sum_{j\in\Bbb I}v_j\pi_{W_j}\Lambda_j^{*} f_j , g\rangle\big\vert\\
&\leq\big(\sum_{j\in\Bbb I}\Vert f_j\Vert^2\big)^{\frac{1}{2}}\sup_{\Vert g\Vert=1}\big(\sum_{j\in\Bbb I}v_j^2\Vert \Lambda_j \pi_{W_j}g\Vert^2\big)^{\frac{1}{2}}\\
&\leq \sqrt{B}\big(\sum_{j\in\Bbb I}\Vert f_j\Vert^2\big)^{\frac{1}{2}}<\infty
\end{align*}
and it follows that $\sum_{j\in\Bbb J}v_j \pi_{W_j}\Lambda_{j}^{*}f_j$ is unconditionally convergent in $H$ (see \cite{diestel} page 58).
\end{proof}
Now, we can define the \textit{synthesis operator} by Proposition \ref{2.2}.
\begin{definition}
Let $\Lambda$ be a g-fusion frame for $H$. Then, the synthesis operator for $\Lambda$ is the operator
\begin{eqnarray*}
T_{\Lambda}:\mathscr{H}_2\longrightarrow H
\end{eqnarray*}
defined by
\begin{eqnarray*}
T_{\Lambda}(\lbrace f_j\rbrace_{j\in\Bbb J})=\sum_{j\in\Bbb J}v_j \pi_{W_j}\Lambda_{j}^{*}f_j.
\end{eqnarray*}
\end{definition}
We say the adjoint $T_{\Lambda}^*$ of the synthesis operator the \textit{analysis operator} and it is defined
by the following Proposition.
\begin{proposition}
Let $\Lambda$ be a g-fusion frame for $H$. Then, the analysis operator
\begin{equation*}
T_{\Lambda}^*:H\longrightarrow\mathscr{H}_2
\end{equation*}
is given by
\begin{equation*}
T_{\Lambda}^*(f)=\lbrace v_j \Lambda_j \pi_{W_j}f\rbrace_{j\in\Bbb J}.
\end{equation*}
\end{proposition}
\begin{proof}
If $f\in H$ and $\lbrace g_j\rbrace_{j\in\Bbb J}\in\mathscr{H}_2$, we have
\begin{align*}
\langle T_{\Lambda}^*(f), \lbrace g_j\rbrace_{j\in\Bbb J}\rangle&=\langle f, T_{\Lambda}\lbrace g_j\rbrace_{j\in\Bbb J}\rangle\\
&=\langle f, \sum_{j\in\Bbb J}v_j \pi_{W_j}\Lambda_{j}^{*}g_j\rangle\\
&=\sum_{j\in\Bbb J}v_j \langle \Lambda_j \pi_{W_j}f, g_j\rangle\\
&=\langle\lbrace v_j \Lambda_j \pi_{W_j}f\rbrace_{j\in\Bbb J}, \lbrace g_j\rbrace_{j\in\Bbb J}\rangle.
\end{align*}
\end{proof}
\begin{theorem}\label{t2}
The following assertions are equivalent:
\begin{enumerate}
\item $\Lambda$ is a g-fusion Bessel sequence for $H$ with bound $B$.
\item The operator
\begin{align*}
T_{\Lambda}&:\mathscr{H}_2\longrightarrow H\\
T_{\Lambda}(\lbrace f_j\rbrace_{j\in\Bbb J})&=\sum_{j\in\Bbb J}v_j \pi_{W_j}\Lambda_{j}^{*}f_j
\end{align*}
is a well-defined and bounded operator with $\Vert T_{\lambda}\Vert\leq \sqrt{B}$.
\item The series
\begin{align*}
\sum_{j\in\Bbb J}v_j \pi_{W_j}\Lambda_{j}^{*}f_j
\end{align*}
 converges for all $\lbrace f_j\rbrace_{j\in\Bbb J}\in\mathscr{H}_2$.
\end{enumerate}
\end{theorem}
\begin{proof}
$\textit{(1)}\Rightarrow\textit{(2)} $. It is clear by Proposition \ref{2.2}.\\
$\textit{(2)}\Rightarrow\textit{(1)}$. Suppose that $T_{\Lambda}$ is a well-defined and bounded operator with $\Vert T_{\lambda}\Vert\leq \sqrt{B}$. Let $\Bbb I$ be a finite subset of $\Bbb J$ and $f\in H$.
Therefore
\begin{align*}
\sum_{j\in\Bbb I}v_j^2\Vert \Lambda_j \pi_{W_j}f\Vert^2&=\sum_{j\in\Bbb I}v_j^2\langle \pi_{W_j}\Lambda^*_j \Lambda_j \pi_{W_j}f, f \rangle\\
&=\langle T_{\Lambda}\lbrace v_j\Lambda_j \pi_{W_j}f\rbrace_{j\in\Bbb I}, f\rangle\\
&\leq \Vert T_{\Lambda}\Vert \Vert v_j\Lambda_j \pi_{W_j}f\Vert \Vert f\Vert\\
&=\Vert T_{\Lambda}\Vert\big(\sum_{j\in\Bbb I}v_j^2\Vert \Lambda_j \pi_{W_j}f\Vert^2\big)^{\frac{1}{2}}\Vert f\Vert.
\end{align*}
Thus, we conclude that
\begin{align*}
\sum_{j\in\Bbb J}v_j^2\Vert \Lambda_j \pi_{W_j}f\Vert^2\leq\Vert T_{\Lambda}\Vert^2 \Vert f\Vert^2\leq B\Vert f\Vert^2.
\end{align*}
$\textit{(1)}\Rightarrow\textit{(3)}$. It is clear.\\
$\textit{(3)}\Rightarrow\textit{(1)}$. Suppose that $\sum_{j\in\Bbb J}v_j \pi_{W_j}\Lambda_{j}^{*}f_j$  converges for all $\lbrace f_j\rbrace_{j\in\Bbb J}\in\mathscr{H}_2$. We define
\begin{align*}
T:&\mathscr{H}_2\longrightarrow H\\
T(\lbrace f_j\rbrace_{j\in\Bbb J})&=\sum_{j\in\Bbb J}v_j \pi_{W_j}\Lambda_{j}^{*}f_j.
\end{align*}
Then, $T$ is well-defined. Let for each $n\in\Bbb N$,
\begin{align*}
T_n:&\mathscr{H}_2\longrightarrow H\\
T_n(\lbrace f_j\rbrace_{j\in\Bbb J})&=\sum_{j=1}^{n}v_j \pi_{W_j}\Lambda_{j}^{*}f_j.
\end{align*}
Let $B_n:=(\sum_{j=1}^{n}\Vert v_j \pi_{W_j}\Lambda_{j}^{*}f_j\Vert^2)^{\frac{1}{2}}$. Since, $\Vert T_n(\lbrace f_j\rbrace_{j\in\Bbb J})\Vert\leq B_n \Vert f_j\lbrace\rbrace_{j\in\Bbb J}\Vert$, then $\lbrace T_n\rbrace$ is a sequence of bounded linear operators which converges pointwise to $T$. Hence, by the Banach-Steinhaus Theorem, $T$ is a bounded operator with
$$\Vert T\Vert\leq\lim\inf\Vert T_n\Vert.$$
So, by Theorem \ref{t2}, $\Lambda$ is a g-fusion Bessel sequence for $H$.
\end{proof}
\begin{corollary}\label{cor}
$\Lambda$ is a g-fusion Bessel sequence for $H$ with bound $B$ if and only if for each finite subset $\Bbb I\subseteq\Bbb J$ and $f_j\in H_j$
$$\Vert\sum_{j\in\Bbb I}v_j \pi_{W_j}\Lambda^*_j f_j\Vert^2\leq B\sum_{j\in\Bbb I}\Vert f_j\Vert^2.$$
\end{corollary}
\begin{proof}
It is an immediate consequence of Theorem \ref{t2} and the proof of Proposition \ref{2.2}.
\end{proof}
Let $\Lambda$ be a g-fusion frame for $H$. The \textit{g-fusion frame operator} is defined by
\begin{align*}
S_{\Lambda}&:H\longrightarrow H\\
S_{\Lambda}f&=T_{\Lambda}T^*_{\Lambda}f.
\end{align*}
Now, for each $f\in H$ we have
$$S_{\Lambda}f=\sum_{j\in\Bbb J}v_j^2 \pi_{W_j}\Lambda^*_j \Lambda_j \pi_{W_j}f$$
and
$$\langle S_{\Lambda}f, f\rangle=\sum_{j\in\Bbb J}v_j^2\Vert \Lambda_j \pi_{W_j}f\Vert^2.$$
Therefore,
$$A I\leq S_{\Lambda}\leq B I.$$
This means that $S_{\Lambda}$ is a bounded, positive and invertible operator (with adjoint inverse). So, we have the reconstruction formula for any $f\in H$:
\begin{equation}\label{3}
f=\sum_{j\in\Bbb J}v_j^2 \pi_{W_j}\Lambda^*_j \Lambda_j \pi_{W_j}S^{-1}_{\Lambda}f
=\sum_{j\in\Bbb J}v_j^2 S^{-1}_{\Lambda}\pi_{W_j}\Lambda^*_j \Lambda_j \pi_{W_j}f.
\end{equation}
\begin{example}
 We introduce a Parseval g-fusion frame for $H$ by the g-fusion frame operator. Assume that $\Lambda=(W_j, \Lambda_j, v_j)$ is a g-fusion frame for $H$. Since $S_{\Lambda}$(or $S_{\Lambda}^{-1}$) is positive in $\mathcal{B}(H)$ and $\mathcal{B}(H)$ is a $C^*$-algebra, then there exist a unique positive square root $S^{\frac{1}{2}}_{\Lambda}$ (or $S^{-\frac{1}{2}}_{\Lambda}$) and they commute with $S_{\Lambda}$ and $S_{\Lambda}^{-1}$. Therefore, each $f\in H$ can be written
\begin{align*}
f&=S^{-\frac{1}{2}}_{\Lambda}S_{\Lambda}S^{-\frac{1}{2}}_{\Lambda}\\
&=\sum_{j\in\Bbb J}v_j^2 S^{-\frac{1}{2}}_{\Lambda}\pi_{W_j}\Lambda^*_j \Lambda_j \pi_{W_j}S^{-\frac{1}{2}}_{\Lambda}f.
\end{align*}
This implies that
\begin{align*}
\Vert f\Vert^2&=\langle f, f\rangle\\
&=\langle\sum_{j\in\Bbb J}v_j^2 S^{-\frac{1}{2}}_{\Lambda}\pi_{W_j}\Lambda^*_j \Lambda_j \pi_{W_j}S^{-\frac{1}{2}}_{\Lambda}f, f\rangle\\
&=\sum_{j\in\Bbb J}v_j^2\Vert \Lambda_j \pi_{W_j}S^{-\frac{1}{2}}_{\Lambda}f\Vert^2\\
&=\sum_{j\in\Bbb J}v_j^2\Vert \Lambda_j \pi_{W_j}S^{-\frac{1}{2}}_{\Lambda}\pi_{S^{-\frac{1}{2}}_{\Lambda}W_j}f\Vert^2,
\end{align*}
this means that $(S^{-\frac{1}{2}}_{\Lambda}W_j, \Lambda_j \pi_{W_j}S^{-\frac{1}{2}}_{\Lambda}, v_j)$ is a Parseval g-fusion frame.
\end{example}
\begin{theorem}\label{2.3}
 $\Lambda$ is a g-fusion frame for $H$ if and only if
 \begin{align*}
T_{\Lambda}&:\mathscr{H}_2\longrightarrow H\\
T_{\Lambda}(\lbrace f_j\rbrace_{j\in\Bbb J})&=\sum_{j\in\Bbb J}v_j \pi_{W_j}\Lambda_{j}^{*}f_j
\end{align*}
is a well-defined, bounded and surjective.
\end{theorem}
\begin{proof}
If $\Lambda$ is a g-fusion frame for $H$, the operator $S_{\Lambda}$ is invertible. Thus, $T_{\Lambda}$ is surjective. Conversely, let $T_{\Lambda}$ be a well-defined, bounded and surjective. Then, by Theorem \ref{t2}, $\Lambda$ is a g-fusion Bessel sequence for $H$. So, $T^{*}_{\Lambda}f=\lbrace v_j \Lambda_j \pi_{W_j}f\rbrace_{j\in\Bbb J}$ for all $f\in H$. Since $T_{\Lambda}$ is surjective, by Lemma \ref{l1}, there exists an operator $T^{\dagger}_{\Lambda}:H\rightarrow\mathscr{H}_2$ such that $(T^{\dagger}_{\Lambda})^*T^*_{\Lambda}=I_{H}$. Now, for each $f\in H$ we have
\begin{align*}
\Vert f\Vert^2&\leq\Vert (T^{\dagger}_{\Lambda})^*\Vert^2 \Vert T^*_{\Lambda}f\Vert^2\\
&=\Vert T^{\dagger}_{\Lambda}\Vert^2\sum_{j\in\Bbb J}v_j^2\Vert \Lambda_j \pi_{W_j}f\Vert^2.
\end{align*}
Therefore, $\Lambda$ is a g-fusion frame for $H$ with lower g-fusion frame bound $\Vert T^{\dagger}\Vert^{-2}$ and upper g-fusion frame $\Vert T_{\Lambda}\Vert^2$.
\end{proof}
\begin{theorem}
$\Lambda$ is a g-fusion frame for $H$ if and only if the operator
$$S_{\Lambda}:f\longrightarrow\sum_{j\in\Bbb J}v_j^2 \pi_{W_j}\Lambda^*_j \Lambda_j \pi_{W_j}f$$
is a well-defined, bounded and surjective.
\end{theorem}
\begin{proof}
The necessity of the statement is clear. Let $S_{\Lambda}$ be a well-defined, bounded and surjective operator. Since $\langle S_{\Lambda}f, f\rangle\geq0$ for all $f\in H$, so $S_{\Lambda}$ is positive. Then
$$\ker S_{\Lambda}=(\mathcal{R}_{S^*_{\Lambda}})^{\dagger}=(\mathcal{R}_{S_{\Lambda}})^{\dagger}=\lbrace 0\rbrace$$
thus, $S_{\Lambda}$ is injective. Therefore, $S_{\Lambda}$ is invertible. Thus, $0\notin\sigma(S_{\Lambda})$. Let $C:=\inf_{\Vert f\Vert=1}\langle S_{\Lambda}f, f\rangle$. By Proposition 70.8 in \cite{he}, we have $C\in\sigma(S_{\Lambda})$. So $C>0$. Now, we can write for each $f\in H$
$$\sum_{j\in\Bbb J}v_j^2\Vert \Lambda_j \pi_{W_j}f\Vert^2=\langle S_{\Lambda}f, f\rangle\geq C\Vert f\Vert^2$$
and
$$\sum_{j\in\Bbb J}v_j^2\Vert \Lambda_j \pi_{W_j}f\Vert^2=\langle S_{\Lambda}f, f\rangle\leq \Vert S_{\Lambda}\Vert \Vert f\Vert^2.$$
It follows that $\Lambda$ is a g-fusion frame for $H$.
\end{proof}
\begin{theorem}
Let $\Lambda:=(W_j, \Lambda_j, v_j)$ and $\Theta:=(W_j, \Theta_j, v_j)$ be two g-fusion Bessel sequence for $H$ with bounds $B_1$ and $B_2$, respectively. Suppose that $T_{\Lambda}$ and $T_{\Theta}$ be their analysis operators such that $T_{\Theta}T^*_{\Lambda}=I_H$. Then, both $\Lambda$ and $\Theta$ are g-fusion frames.
\end{theorem}
\begin{proof}
For each $f\in H$ we have
\begin{align*}
\Vert f\Vert^4&=\langle f, f\rangle^2\\
&=\langle T^*_{\Lambda}f, T^*_{\Theta}f\rangle^2\\
&\leq\Vert T^*_{\Lambda}f\Vert^2 \Vert T^*_{\Theta}f\Vert^2\\
&=\big(\sum_{j\in\Bbb I}v_j^2\Vert \Lambda_j \pi_{W_j}f\Vert^2\big)\big(\sum_{j\in\Bbb I}v_j^2\Vert \Theta_j \pi_{W_j}f\Vert^2\big)\\
&\leq\big(\sum_{j\in\Bbb I}v_j^2\Vert \Lambda_j \pi_{W_j}f\Vert^2\big) B_2 \Vert f\Vert^2,
\end{align*}
thus, $B_2^{-1}\Vert f\Vert^2\leq\sum_{j\in\Bbb I}v_j^2\Vert \Lambda_j \pi_{W_j}f\Vert^2$. This means that $\Lambda$ is a g-fusion frame for $H$. Similarly, $\Theta$ is a g-fusion frame with the lower bound $B_1^{-1}$.
\end{proof}
\section{Dual g-Fusion Frames}
For definition of the dual g-fusion frames, we need the following theorem.
\begin{theorem}\label{dual}
Let $\Lambda=(W_j, \Lambda_j, v_j)$ be a g-fusion frame for $H$. Then $(S^{-1}_{\Lambda}W_j, \Lambda_j \pi_{W_j}S_{\Lambda}^{-1}, v_j)$ is a g-fusion frame for $H$.
\end{theorem}
\begin{proof}
 Let $A,B$ be the g-fusion frame bounds of $\Lambda$ and $f\in H$, then
\begin{align*}
\sum_{j\in\Bbb J}v^2_j\Vert \Lambda_{j}\pi_{W_j}S_{\Lambda}^{-1}\pi_{S^{-1}_{\Lambda}W_j}f\Vert^2&=\sum_{j\in\Bbb J}v^2_j\Vert \Lambda_{j}\pi_{W_j}S_{\Lambda}^{-1}f\Vert^2\\
&\leq B\Vert S_{\Lambda}^{-1}\Vert^2 \Vert f\Vert^2.
\end{align*}
Now, to get the lower bound, by using (\ref{3}) we can write
\begin{align*}
\Vert f\Vert^4&=\big\vert\langle\sum_{j\in\Bbb J}v_j^2\pi_{W_j}\Lambda^*_J\Lambda_j\pi_{W_j}S^{-1}_{\Lambda}f, f\rangle\big\vert^2\\
&=\big\vert\sum_{j\in\Bbb J}v_j^2\langle\Lambda_j\pi_{W_j}S^{-1}_{\Lambda}f, \Lambda_j\pi_{W_j}f\rangle\big\vert^2\\
&\leq\sum_{j\in\Bbb J}v_j^2\Vert \Lambda_j\pi_{W_j}S^{-1}_{\Lambda}f\Vert^2 \sum_{j\in\Bbb J}v_j^2\Vert\Lambda_j\pi_{W_j}f\Vert^2\\
&\leq \sum_{j\in\Bbb J}v_j^2\Vert\Lambda_j\pi_{W_j} S_{\Lambda}^{-1}\pi_{S^{-1}_{\Lambda}W_j}f\Vert^2\big(B\Vert f\Vert^2\big),
\end{align*}
therefore
\begin{align*}
B^{-1}\Vert f\Vert^2\leq \sum_{j\in\Bbb J}v_j^2\Vert\Lambda_j\pi_{W_j} S_{\Lambda}^{-1}\pi_{S^{-1}_{\Lambda}W_j}f\Vert^2.
\end{align*}
\end{proof}
Now, by Theorem \ref{dual}, $\tilde{\Lambda}=(S^{-1}_{\Lambda}W_j, \Lambda_j\pi_{W_j} S_{\Lambda}^{-1}, v_j)$ is a g-fusion frame for $H$. Then, $\tilde{\Lambda}$ is called the \textit{(canonical) dual g-fusion frame}  of $\Lambda$. Let $S_{\tilde{\Lambda}}=T_{\tilde{\Lambda}}T^*_{\tilde{\Lambda}}$ is the g-fusion frame operator of  $\tilde{\Lambda}$. Then, for each $f\in H$ we get
$$T^*_{\tilde{\Lambda}}f=\lbrace v_j\Lambda_j\pi_{W_j}S^{-1}_{\Lambda}\pi_{S^{-1}_{\Lambda W_j}}f\rbrace=\lbrace v_j\Lambda_j\pi_{W_j} S^{-1}_{\Lambda}f\rbrace=T^*_{\Lambda}(S^{-1}_{\Lambda}f),$$
so $T_{\Lambda}T^*_{\tilde{\Lambda}}=I_H$. Also, we have for each $f\in H$,
\begin{align*}
\langle S_{\tilde{\Lambda}}f, f\rangle&=\sum_{j\in\Bbb J}v_j^2\Vert\Lambda_j\pi_{W_j} S_{\Lambda}^{-1}\pi_{S^{-1}_{\Lambda}W_j}f\Vert^2\\
&=\sum_{j\in\Bbb J}v_j^2\Vert\Lambda_j \pi_{W_j}S_{\Lambda}^{-1}f\Vert^2\\
&=\langle S_{\Lambda}(S_{\Lambda}^{-1}f), S_{\Lambda}^{-1}f\rangle\\
&=\langle S_{\Lambda}^{-1}f, f\rangle
\end{align*}
thus, $S_{\tilde{\Lambda}}=S_{\Lambda}^{-1}$ and by (\ref{3}), we get for each $f\in H$
\begin{align}\label{frame}
f=\sum_{j\in\Bbb J}v_j^2\pi_{W_j}\Lambda^*_j\tilde{\Lambda_j}\pi_{\tilde{W_j}}f=
\sum_{j\in\Bbb J}v_j^2\pi_{\tilde{W_j}}\tilde{\Lambda_j}^*\Lambda_j\pi_{W_j}f,
\end{align}
where $\tilde{W_j}:=S^{-1}_{\Lambda}W_j  \ , \ \tilde{\Lambda_j}:=\Lambda_j \pi_{W_j}S_{\Lambda}^{-1}.$

The following Theorem  shows that the canonical dual g-fusion frame
gives rise to expansion coefficients with the minimal norm.
\begin{theorem}\label{min}
Let $\Lambda$ be a g-fusion frame with canonical dual $\tilde{\Lambda}$. 
For each $g_j\in H_j$, put $f=\sum_{j\in\Bbb J}v_j^2\pi_{W_j}\Lambda^*_j g_j$. Then
$$\sum_{j\in\Bbb J}\Vert g_j\Vert^2=\sum_{j\in\Bbb J}v_j^2\Vert \tilde{\Lambda_j}\pi_{\tilde{W_j}}f\Vert^2+\sum_{j\in\Bbb J}\Vert g_j-v_j^2\tilde{\Lambda_j}\pi_{\tilde{W_j}}f\Vert^2.$$
\end{theorem}
\begin{proof}
We can write again
\begin{align*}
\sum_{j\in\Bbb J}v_j^2\Vert \tilde{\Lambda_j}\pi_{\tilde{W_j}}f\Vert^2&=\langle f, S^{-1}_{\Lambda}f\rangle\\
&=\sum_{j\in\Bbb J}v_j^2\langle\pi_{W_j}\Lambda^*_j g_j, S_{\Lambda}^{-1}f \rangle\\
&=\sum_{j\in\Bbb J}v_j^2\langle g_j, \Lambda_j\pi_{W_j}S_{\Lambda}^{-1}f \rangle\\
&=\sum_{j\in\Bbb J}v_j^2\langle g_j, \tilde{\Lambda_j}\pi_{\tilde{W_j}} f \rangle.
\end{align*}
Therefore, $\mbox{Im}\Big(\sum_{j\in\Bbb J}v_j^2\langle g_j, \tilde{\Lambda_j}\pi_{\tilde{W_j}} f \rangle\Big)=0$. So
\begin{align*}
\sum_{j\in\Bbb J}\Vert g_j-v_j^2\tilde{\Lambda_j}\pi_{\tilde{W_j}}f\Vert^2=\sum_{j\in\Bbb J}\Vert g_j\Vert^2 -2\sum_{j\in\Bbb J}v_j^2\langle g_j, \tilde{\Lambda_j}\pi_{\tilde{W_j}} f \rangle+\sum_{j\in\Bbb J}v_j^2\Vert \tilde{\Lambda_j}\pi_{\tilde{W_j}}f\Vert^2
\end{align*}
and  the proof completes.
\end{proof}
\section{Gf-Complete and g-Fusion Frame Sequences}
\begin{definition}
We say that  $(W_j, \Lambda_j)$ is  \textit{gf-complete} , if
$\overline{\mbox{span}}\lbrace \pi_{W_j}\Lambda^*_j H_j\rbrace=H.$
\end{definition}
Now, it is easy to check that $(W_j, \Lambda_j)$ is gf-complete if and only if
$$\lbrace f: \ \Lambda_j \pi_{W_j}f=0 , \ j\in\Bbb J\rbrace=\lbrace 0\rbrace.$$
\begin{proposition}\label{p3}
If $\Lambda=(W_j, \Lambda_j, v_j)$ is a g-fusion frame for $H$, then $(W_j, \Lambda_j)$ is a gf-complete.
\end{proposition}
\begin{proof}
Let $f\in(\mbox{span}\lbrace \pi_{W_j}\Lambda^*_j H_j\rbrace)^{\perp}\subseteq H$. For each $j\in\Bbb J$ and $g_j\in H_j$ we have
$$\langle \Lambda_j\pi_{W_j}f, g_j\rangle=\langle f, \pi_{W_j}\Lambda^*_j g_j\rangle=0,$$
so, $\Lambda_j\pi_{W_j}f=0$ for all $j\in\Bbb J$. Since $\Lambda$ is a g-fusion frame for $H$, then $\Vert f\Vert=0$. Thus $f=0$ and we get $(\mbox{span}\lbrace \pi_{W_j}\Lambda^*_j H_j\rbrace)^{\perp}=\lbrace0\rbrace$.
\end{proof}
In the following, we want to check that if a member is removed from a g-fusion frame, will the new set remain a g-fusion frame or not?
\begin{theorem}\label{del}
Let $\Lambda=(W_j, \Lambda_j, v_j)$ be a g-fusion frame for $H$ with bounds $A, B$ and $\tilde{\Lambda}=(S^{-1}_{\Lambda}W_j, \Lambda_j \pi_{W_j}S^{-1}_{\Lambda}, v_j)$ be a canonical dual g-fusion frame. Suppose that $j_0\in\Bbb J$.
\begin{enumerate}
\item If there is a $g_0\in H_{j_0}\setminus\lbrace 0\rbrace$ such that $\tilde\Lambda_{j_0}\pi_{\tilde{W}_{j_0}}\pi_{W_{j_0}}\Lambda^*_{j_0}g_0=g_0$ and $v_{j_0}=1$, then $(W_j, \Lambda_j)_{j\neq j_0}$ is not gf-complete in $H$.
\item If there is a $f_0\in H_{j_0}\setminus\lbrace0\rbrace$ such that $\pi_{W_{j_0}}\Lambda^*_{j_0}\tilde\Lambda_{j_0}\pi_{\tilde{W}_{j_0}}f_0=f_0$ and $v_{j_0}=1$, then $(W_j, \Lambda_j)_{j\neq j_0}$ is not gf-complete in $H$.
\item If $I-\Lambda_{j_0}\pi_{W_{j_0}}\pi_{\tilde{W_{j_0}}}\tilde{\Lambda}^*_{j_0}$ is bounded invertible on $H_{j_0}$, then $(W_j, \Lambda_j, v_j)_{j\neq j_0}$ is a g-fusion frame for $H$.
\end{enumerate}
\end{theorem}
\begin{proof}
\textit{(1).} Since $\pi_{W_{j_0}}\Lambda^*_{j_0}g_0\in H$, then by (\ref{frame}),
$$\pi_{W_{j_0}}\Lambda^*_{j_0}g_0=\sum_{j\in\Bbb J}v_j^2\pi_{W_j}\Lambda^*_j\tilde{\Lambda_j}\pi_{\tilde{W_j}}\pi_{W_{j_0}}\Lambda^*_{j_0}g_0.$$
So,
$$\sum_{j\neq j_0}v_j^2\pi_{W_j}\Lambda^*_j\tilde{\Lambda_j}\pi_{\tilde{W_j}}\pi_{W_{j_0}}\Lambda^*_{j_0}g_0=0.$$
Let $u_{j_0, j}:=\delta_{j_0, j}g_0$, thus
$$\pi_{W_{j_0}}\Lambda^*_{j_0}g_0=\sum_{j\in\Bbb J}v_j^2\pi_{W_j}\Lambda^*_j u_{j_0, j}.$$
Then, by Theorem \ref{min}, we have
$$\sum_{j\in\Bbb J}\Vert u_{j_0, j}\Vert^2=\sum_{j\in\Bbb J}v_j^2\Vert \tilde{\Lambda}_j\pi_{\tilde{W_j}}\pi_{W_{j_0}}\Lambda^*_{j_0}g_0\Vert^2+\sum_{j\in\Bbb J}\Vert v_j^2\tilde{\Lambda}_j\pi_{\tilde{W_j}}\pi_{W_{j_0}}\Lambda^*_{j_0}g_0-u_{j_0, j}\Vert^2.$$
Consequently,
$$\Vert g_0\Vert^2=\Vert g_0\Vert^2+2\sum_{j\neq j_0}v_j^2\Vert \tilde{\Lambda}_j\pi_{\tilde{W_j}}\pi_{W_{j_0}}\Lambda^*_{j_0}g_0\Vert^2$$
and we get $\tilde{\Lambda}_j\pi_{\tilde{W_j}}\pi_{W_{j_0}}\Lambda^*_{j_0}g_0=0$.
Therefore,
$$\Lambda_j\pi_{W_j}S^{-1}_{\Lambda}\pi_{W_{j_0}}\Lambda^*_{j_0}g_0=\tilde{\Lambda}_j\pi_{\tilde{W_j}}\pi_{W_{j_0}}\Lambda^*_{j_0}g_0=0.$$
But, $g_0=\tilde\Lambda_{j_0}^*\pi_{\tilde{W}_{j_0}}\pi_{W_{j_0}}\Lambda^*_{j_0}g_0=\Lambda_{j_0}\pi_{W_{j_0}}S^{-1}_{\Lambda}\pi_{W_{j_0}}\Lambda^*_{j_0}g_0\neq0$, which implies that $S^{-1}_{\Lambda}\pi_{W_{j_0}}\Lambda^*_{j_0}g_0\neq0$ and this means that $(W_j, \Lambda_j)_{j\neq j_0}$ is not gf-complete in $H$.

\textit{(2).} Since $\pi_{W_{j_0}}\Lambda^*_{j_0}\tilde\Lambda_{j_0}\pi_{\tilde{W}_{j_0}}f_0=f_0\neq0$, we obtain $\tilde\Lambda_{j_0}\pi_{\tilde{W}_{j_0}}f_0\neq0$ and
$$\tilde\Lambda_{j_0}\pi_{\tilde{W}_{j_0}}\pi_{W_{j_0}}\Lambda^*_{j_0}\tilde\Lambda_{j_0}\pi_{\tilde{W}_{j_0}}f_0=\tilde\Lambda_{j_0}\pi_{\tilde{W}_{j_0}}f_0.$$
Now, the conclusion follows from \textit{(1)}.

\textit{(3)}. Using (\ref{frame}), we have for any $f\in H$
$$\Lambda_{j_0}\pi_{W_{j_0}}f=\sum_{j\in\Bbb J}v_j^2\Lambda_{j_0}\pi_{W_{j_0}}\pi_{\tilde{W_j}}\tilde{\Lambda}^*_j\Lambda_j\pi_{W_j}f.$$
So,
\begin{equation}\label{com}(I-\Lambda_{j_0}\pi_{W_{j_0}}\pi_{\tilde{W_{j_0}}}\tilde{\Lambda}^*_{j_0})\Lambda_{j_0}\pi_{W_{j_0}}f=\sum_{j\neq j_0}v_j^2\Lambda_{j_0}\pi_{W_{j_0}}\pi_{\tilde{W_j}}\tilde{\Lambda}^*_j\Lambda_j\pi_{W_j}f.
\end{equation}
On the other hand, we can write
\begin{small}
\begin{align*}
\big\Vert\sum_{j\neq j_0}v_j^2\Lambda_{j_0}\pi_{W_{j_0}}\pi_{\tilde{W_j}}\tilde{\Lambda}^*_j\Lambda_j\pi_{W_j}f\big\Vert^2&=\sup_{\Vert g\Vert=1}\big\vert\sum_{j\neq j_0}v_j^2\big\langle \Lambda_j\pi_{W_j}f, \tilde{\Lambda}_j\pi_{\tilde{W}_j}\pi_{W_{j_0}}\Lambda^*_{j_0}g \big\rangle\big\vert^2\\
&\leq\big(\sum_{j\neq j_0}v_j^2\Vert \Lambda_j\pi_{W_j}f\Vert^2\big)\sup_{\Vert g\Vert=1}\sum_{j\in\Bbb J}v_j^2\Vert \tilde{\Lambda}_j\pi_{\tilde{W}_j}\pi_{W_{j_0}}\Lambda^*_{j_0}g\Vert^2\\
&\leq\tilde{B}\Vert\Lambda_{j_0}\Vert^2(\sum_{j\neq j_0}v_j^2\Vert \Lambda_j\pi_{W_j}f\Vert^2)
\end{align*}
\end{small}
where, $\tilde{B}$ is the upper bound of $\tilde\Lambda$. Now, by (\ref{com}), we have
\begin{equation*}
\Vert\Lambda_{j_0}\pi_{W_{j_0}}f\Vert^2\leq\Vert(I-\Lambda_{j_0}\pi_{W_{j_0}}\pi_{\tilde{W_{j_0}}}\tilde{\Lambda}^*_{j_0})^{-1}\Vert^2 \tilde{B}\Vert\Lambda_{j_0}\Vert^2(\sum_{j\neq j_0}v_j^2\Vert \Lambda_j\pi_{W_j}f\Vert^2).
\end{equation*}
Therefore, there is a number $C>0$ such that
$$\sum_{j\in\Bbb J}v_j^2\Vert \Lambda_j\pi_{W_j}f\Vert^2\leq C\sum_{j\neq j_0}v_j^2\Vert \Lambda_j\pi_{W_j}f\Vert^2$$
and we conclude for each $f\in H$
$$\frac{A}{C}\Vert f\Vert^2\leq\sum_{j\neq j_0}v_j^2\Vert \Lambda_j\pi_{W_j}f\Vert^2\leq B\Vert f\Vert^2.$$
\end{proof}
\begin{theorem}
$\Lambda$ is a g-fusion frame for $H$ with bounds $A,B$ if and only if the following two conditions are satisfied:
\begin{enumerate}
\item[(I)] The pair $(W_j, \Lambda_j)$ is gf-complete.
\item[(II)] The operator
$$T_{\Lambda}: \lbrace f_j\rbrace_{j\in\Bbb J}\mapsto \sum_{j\in\Bbb J}v_j\pi_{W_j}\Lambda_j^* f_j$$
is a well-defined from $\mathscr{H}_2$ into $H$ and for each $\lbrace f_j\rbrace_{j\in\Bbb J}\in\mathcal{N}^{\perp}_{T_{\Lambda}}$,
\begin{equation}\label{e7}
A\sum_{j\in\Bbb J}\Vert f_j\Vert^2\leq \Vert T_{\Lambda}\lbrace f_j\rbrace_{j\in\Bbb J}\Vert^2\leq B\sum_{j\in\Bbb J}\Vert f_j\Vert^2.
\end{equation}
\end{enumerate}
\end{theorem}
\begin{proof}
First, suppose that $\Lambda$ is a g-fusion frame. By Proposition \ref{p3}, (I) is satisfied. By Theorem \ref{t2}, $T_{\Lambda}$ is a well-defined from $\mathscr{H}_2$ into $H$ and $\Vert T_{\Lambda}\Vert^2\leq B$. Now, the right-hand inequality in (\ref{e7}) is proved.

By Theorem \ref{2.3}, $T_{\Lambda}$ is surjective. So, $\mathcal{R}_{T^*_{\Lambda}}$ is closed. Thus
$$\mathcal{N}^{\perp}_{T_{\Lambda}}=\overline{\mathcal{R}_{T^*_{\Lambda}}}=\mathcal{R}_{T^*_{\Lambda}}.$$
Now, if $\lbrace f_j\rbrace_{j\in\Bbb J}\in\mathcal{N}^{\perp}_{T_{\Lambda}}$, then
$$\lbrace f_j\rbrace_{j\in\Bbb J}=T^*_{\Lambda}g=\lbrace v_j\Lambda_j \pi_{W_j}g\rbrace_{j\in\Bbb J}$$
for some $g\in H$. Therefore
\begin{align*}
(\sum_{j\in\Bbb J}\Vert f_j\Vert^2)^2&=(\sum_{j\in\Bbb J}v^2_j\Vert \Lambda_j \pi_{W_j}g\Vert^2)^2
\vert\langle S_{\Lambda}(g), g\rangle\vert^2\\
&\leq\Vert S_{\Lambda}(g)\Vert^2 \Vert g\Vert^2\\
&\leq\Vert S_{\Lambda}(g)\Vert^2  \big(\frac{1}{A}\sum_{j\in\Bbb J}v^2_j\Vert \Lambda_j \pi_{W_j}g\Vert^2\big).
\end{align*}
This implies that
$$A\sum_{j\in\Bbb J}\Vert f_j\Vert^2\leq\Vert S_{\Lambda}(g)\Vert^2=\Vert T_{\Lambda}\lbrace f_j\rbrace_{j\in\Bbb J}\Vert^2$$
and (II) is proved.

Conversely, Let $(W_j, \Lambda_j)$ be gf-complete and inequality
(\ref{e7}) is satisfied. Let $\lbrace t_j\rbrace_{j\in\Bbb
J}=\lbrace f_j\rbrace_{j\in\Bbb J}+\lbrace g_j\rbrace_{j\in\Bbb J}$,
where $\lbrace f_j\rbrace_{j\in\Bbb J}\in\mathcal{N}_{T_{\Lambda}}$
and $\lbrace g_j\rbrace_{j\in\Bbb
J}\in\mathcal{N}_{T_{\Lambda}}^{\perp}$. We get
\begin{align*}
\Vert T_{\Lambda}\lbrace t_j\rbrace_{j\in\Bbb J}\Vert^2&=\Vert T_{\Lambda}\lbrace g_j\rbrace_{j\in\Bbb J}\Vert^2\\
&\leq B\sum_{j\in\Bbb J}\Vert g_j\Vert^2\\
&\leq B\Vert \lbrace f_j\rbrace+\lbrace g_j\rbrace\Vert^2\\
&=B\Vert\lbrace t_j\rbrace_{j\in\Bbb J}\Vert^2.
\end{align*}
Thus, $\Lambda$ is a g-fusion Bessel sequence.

Assume that $\lbrace y_n\rbrace$ is a sequence of members of $\mathcal{R}_{T_{\Lambda}}$ such that $y_n\rightarrow y$ for some $y\in H$. So, there is a $\lbrace x_n\rbrace\in\mathcal{N}_{T_{\Lambda}}$ such that $T_{\Lambda}\{x_n\}=y_n$. By (\ref{e7}), we obtain
\begin{align*}
A\Vert\lbrace x_n-x_m\rbrace\Vert^2&\leq\Vert T_{\Lambda}\lbrace x_n-x_m\rbrace\Vert^2\\
&=\Vert T_{\Lambda}\lbrace x_n\rbrace -T_{\Lambda}\lbrace x_m\rbrace\Vert^2\\
&=\Vert y_n-y_m\Vert^2.
\end{align*}
Therefore, $\lbrace x_n\rbrace$ is a Cauchy sequence in $\mathscr{H}_2$. Therefore $\lbrace x_n\rbrace$ converges to some $x\in \mathscr{H}_2$, which by continuity of $T_{\Lambda}$,  we have $y=T_{\Lambda}(x)\in\mathcal{R}_{T_{\Lambda_j}}$. Hence $\mathcal{R}_{T_{\Lambda}}$ is closed. Since $\mbox{span}\lbrace \pi_{W_j}\Lambda_{j}^{\ast}(H_j)\rbrace\subseteq\mathcal{R}_{T_{\Lambda}}$,  by (I) we get $\mathcal{R}_{T_{\Lambda}}=H$.

 Let $T_{\Lambda}^\dagger$ denotes the pseudo-inverse of $T_{\Lambda}$. By Lemma \ref{Ru}(4),  $T_{\Lambda}T_{\Lambda}^{\dagger}$ is the orthogonal projection onto $\mathcal{R}_{T_{\Lambda}}=H$. Thus for any $\lbrace f_j\rbrace_{j\in\Bbb J}\in\mathscr{H}_2 $,
\begin{eqnarray*}
A\Vert T_{\Lambda}^{\dagger}T_{\Lambda}\lbrace f_j\rbrace\Vert^2\leq\Vert T_{\Lambda}T_{\Lambda}^{\dagger}T_{\Lambda}\lbrace f_j\rbrace \Vert^2=\Vert T_{\Lambda}\lbrace f_j\rbrace\Vert^2.
\end{eqnarray*}
By  Lemma \ref{Ru} (4),  $\mathcal{N}_{{T}_{\Lambda}^{\dagger}}=\mathcal{R}^{\bot}_{T_{\Lambda}}$, therefore
\begin{eqnarray*}
\Vert T_{\Lambda}^\dagger\Vert^2\leq\frac{1}{A}.
\end{eqnarray*}
Also by  Lemma \ref{Ru}(2), we have
$$ \Vert(T_{\Lambda}^\ast)^{\dagger}\Vert^2\leq\frac{1}{A}.$$
But  $(T_{\Lambda}^\ast)^{\dagger}T_{\Lambda}^\ast$ is the
orthogonal projection onto
\begin{eqnarray*}
\mathcal{R}_{(T_{\Lambda}^\ast)^\dagger}=\mathcal{R}_{(T_{\Lambda}^\dagger)^\ast}=\mathcal{N}_{T_{\Lambda}^\dagger}^{\bot}=\mathcal{R}_{T_{\Lambda}}=H.
\end{eqnarray*}
So, for all $f\in H$
\begin{align*}
\Vert f\Vert^2&=\Vert(T_{\Lambda}^\ast)^{\dagger}T_{\Lambda}^\ast f\Vert^2\\
&\leq \frac{1}{A}\Vert T_{\Lambda}^\ast f\Vert^2\\
&=\frac{1}{A}\sum_{j\in\Bbb J}v^2_j\Vert \Lambda_j \pi_{W_j}f\Vert^2.
\end{align*}
This implies that $\Lambda$ satisfies the lower g-fusion frame condition.
\end{proof}
Now, we can define a g-fusion frame sequence in the Hilbert space.
\begin{definition}
We say that $\Lambda$ is a \textit{g-fusion frame sequence}  if it is a g-fusion frame for    $\overline{\mbox{span}}\lbrace \pi_{W_j}\Lambda^*_j H_j\rbrace$.
\end{definition}
\begin{theorem}\label{2.6}
$\Lambda$ is a g-fusion frame sequence if and only if the operator
 \begin{align*}
T_{\Lambda}&:\mathscr{H}_2\longrightarrow H\\
T_{\Lambda}(\lbrace f_j\rbrace_{j\in\Bbb J})&=\sum_{j\in\Bbb J}v_j \pi_{W_j}\Lambda_{j}^{*}f_j
\end{align*}
is a well-defined and has closed range.
\end{theorem}
\begin{proof}
By Theorem \ref{2.3}, it is enough to prove that if $T_{\lambda}$ has closed range, then $\overline{\mbox{span}}\lbrace \pi_{W_j}\Lambda^*_j H_j\rbrace=\mathcal{R}_{T_{\Lambda}}$.
Let $f\in\overline{\mbox{span}}\lbrace \pi_{W_j}\Lambda^*_j H_j\rbrace$, then
$$f=\lim_{n\rightarrow\infty}g_n , \ \ \ g_n\in\mbox{span}\lbrace \pi_{W_j}\Lambda^*_j H_j\rbrace\subseteq \mathcal{R}_{T_{\Lambda}}=\overline{\mathcal{R}}_{T_{\Lambda}}.$$
Therefore, $\overline{\mbox{span}}\lbrace \pi_{W_j}\Lambda^*_j H_j\rbrace\subseteq\overline{\mathcal{R}}_{T_{\Lambda}}=\mathcal{R}_{T_{\Lambda}}$. On the other hand, if $f\in\mathcal{R}_{T_{\Lambda}}$, then
$$f\in\mbox{span}\lbrace \pi_{W_j}\Lambda^*_j H_j\rbrace\subseteq\overline{\mbox{span}}\lbrace \pi_{W_j}\Lambda^*_j H_j\rbrace$$
 and the proof is completed.
\end{proof}
\begin{theorem}
 $\Lambda$ is a g-fusion frame sequence if and only if
\begin{equation}\label{4}
f \longmapsto \lbrace v_j \Lambda_j \pi_{W_j}f\rbrace_{j\in\Bbb J}
\end{equation}
defines a map from $H$ onto a closed subspace of $\mathscr{H}_2$.
\end{theorem}
\begin{proof}
Let $\Lambda$ be a g-fusion frame sequence. Then, by Theorem \ref{2.6}, $T_{\lambda}$ is well-defined and $\mathcal{R}_{T_{\Lambda}}$ is closed. So, $T^*_{\Lambda}$ is well-defined and has closed range. Conversely, by hypothesis, for all $f\in H$
$$\sum_{j\in\Bbb J}\Vert v_j \Lambda_j \pi_{W_j}f\Vert^2<\infty.$$
Let
$$B:=\sup\big\lbrace \sum_{j\in\Bbb J}v_j^2\Vert \Lambda_j \pi_{W_j}f\Vert^2 : \ \ f\in H, \ \Vert f\Vert=1\big\rbrace$$
and suppose that $g_j\in H_j$ and $\Bbb I\subseteq\Bbb J$ be finite. We can write
\begin{align*}
\Vert\sum_{j\in\Bbb I}v_j \pi_{W_j}\Lambda^*_j g_j\Vert^2&=\Big(\sup_{\Vert f\Vert=1}\big\vert\langle\sum_{j\in\Bbb I}v_j \pi_{W_j}\Lambda^*_j g_j, f\rangle\big\vert\Big)^2\\
&\leq\Big(\sup_{\Vert f\Vert=1}\sum_{j\in\Bbb I}v_j\big\vert\langle  g_j, \Lambda_j \pi_{W_j}f\rangle\big\vert\Big)^2\\
&\leq\big(\sum_{j\in\Bbb I}\Vert g_j\Vert^2\big)\big(\sup_{\Vert f\Vert=1}\sum_{j\in\Bbb I}v_j^2\Vert \Lambda_j \pi_{W_j}f\Vert^2\big)\\
&\leq B\big(\sum_{j\in\Bbb I}\Vert g_j\Vert^2\big)
\end{align*}
Thus, by Corollary \ref{cor}, $\Lambda$ is a g-fusion Bessel sequence for $H$. Therefore, $T_{\Lambda}$ is well-defined and bounded. Furthermore, if the range of the map  (\ref{4}) is closed, the same is true for $T_{\Lambda}$. So, by Theorem \ref{2.6}, $\Lambda$ is a g-fusion frame sequence.
\end{proof}
\begin{theorem}
Let $\Lambda=(W_j, \Lambda_j, v_j)$ be a g-fusion frame sequence Then, it is a g-fusion frame for $H$ if and only if the map
\begin{equation}\label{5}
f \longmapsto \lbrace v_j \Lambda_j \pi_{W_j}f\rbrace_{j\in\Bbb J}
\end{equation}
from $H$ onto a closed subspace of $\mathscr{H}_2$ be injective.
\end{theorem}
\begin{proof}
Suppose that the map (\ref{5}) is injective and $v_j \Lambda_j \pi_{W_j}f=0$ for all $j\in\Bbb J$. Then, the value of the map at $f$ is zero. So, $f=0$. This means that $(W_j, \Lambda_j)$ is gf-complete. Since, $\Lambda$ is a g-fusion frame sequence,  so, it is a g-fusion frame for $H$.

The converse is evident.
\end{proof}
\begin{theorem}
Let $\Lambda$ be a g-fusion frame for $H$ and $u\in\mathcal{B}(H)$. Then $\Gamma:=(uW_j, \Lambda_j u^*, v_j)$ is a g-fusion frame sequence if and only if $u$ has closed range.
\end{theorem}
\begin{proof}
Assume that $\Gamma$ is a g-fusion frame sequence. So, by Theorem \ref{2.6}, $T_{\Lambda u^*}$ is a well-defined operator from $\mathscr{H}_2$ into $H$ with closed range. If $\lbrace f_j\rbrace_{j\in\Bbb J}\in\mathscr{H}_2$, then
\begin{align*}
uT_{\Lambda}\lbrace f_j\rbrace_{j\in\Bbb J}&=\sum_{j\in\Bbb J}v_ju\pi_{W_j}\Lambda_j^* f_j\\
&=\sum_{j\in\Bbb J}v_j\pi_{uW_j}u\Lambda_j^* f_j\\
&=\sum_{j\in\Bbb J}v_j\pi_{uW_j}(\Lambda_j u^*)^* f_j\\
&=T_{\Lambda u^*}\lbrace f_j\rbrace_{j\in\Bbb J},
\end{align*}
therefore $uT_{\Lambda}=T_{\Lambda u^*}$. Thus $uT_{\Lambda}$ has closed range too. Let $y\in\mathcal{R}_u$, then there is $x\in H$ such that $u(x)=y$. By Theorem \ref{2.3}, $T_{\Lambda}$ is surjective, so there exist $\{f_j\}_{j\in\Bbb J}\in\mathscr{H}_2$ such that 
$$y=u(T_{\Lambda}\{f_j\}_{j\in\Bbb J}).$$
Thus, $\mathcal{R}_{u}=\mathcal{R}_{uT_{\Lambda}}$  and $u$ has closed range.

For the opposite implication, let
\begin{align*}
T_{\Lambda u^*}:&\mathscr{H}_2\longrightarrow H\\
T_{\Lambda u^*}\lbrace f_j\rbrace_{j\in\Bbb J}&=\sum_{j\in\Bbb J}v_j\pi_{uW_j}(\Lambda_j u^*)^* f_j.
\end{align*}
Hence, $T_{\Lambda u^*}=uT_{\Lambda}$. Since, $T_{\Lambda}$ is surjective, so $T_{\Lambda u^*}$ has closed range and by Theorem \ref{t2}, is well-defined. Therefore, by Theorem \ref{2.6}, the proof is completed.
\end{proof}
\section{Conclusions}
In this paper, we could transfer some common properties in general frames to g-fusion frames with the definition of the g-fusion frames and their operators. Afterward, we reviewed a basic theorem about deleting a member in Theorem \ref{del} with the definition of the dual g-fusion frames and the gf-completeness. In this theorem,  the defined operator in part  \textit{3} could be replaced  by some other operators which are the same as the parts  \textit{1} and \textit{2}; this is an open problem at the moment. Eventually, the g-fusion frame sequences  and their relationship with the closed range operators were defined and presented.


\begin{thebibliography}{99}

\bibitem{aklr} H. Blocsli,  H. F. Hlawatsch,  and H. G. Fichtinger, 
 \textit{Frame-Theoretic analysis of oversampled filter bank},
 IEEE Trans. Signal Processing \textbf{46} (12), 3256- 3268, 1998.
 
\bibitem {bbk} E. J. Candes,  and D. L. Donoho, 
\textit{New tight frames of curvelets and optimal representation of objects with
  piecwise $C^2$ singularities.}
 Comm. Pure and App. Math. \textbf{57}(2), 219-266, 2004.
 
 \bibitem {caz} P. G. Casazza,  and O. Christensen, 
 \textit{Perturbation of Operators and Application to Frame Theory.}
 J. Fourier Anal. Appl. \textbf{3}, 543-557, 1997.
 
\bibitem {ck} P. G. Casazza,  and G. Kutyniok, 
 \textit{Frames of Subspaces.}
 Contemp. Math. 345, 87-114, 2004.
 
 \bibitem {ck2} P. G. Casazza, G. Kutyniok,  and S. Li, 
 \textit{Fusion Frames and distributed processing.}
 Appl. comput. Harmon. Anal. \textbf{25}(1), 114-132, 2008.

\bibitem {ch} O. Christensen, 
\textit{Frames and Bases: An Introductory Course (Applied and Numerical Harmonic Analysis).}
 Birkhauser,  2008.
 
\bibitem {diestel} J. Diestel, 
\textit{Sequences and series in Banach spaces.}
Springer-Verlag, New York, 1984.

\bibitem {ds} R.J. Duffin,  and  A. C. Schaeffer, 
\textit{A class of nonharmonik Fourier series.}
Trans. Amer. Math. Soc. \textbf{72}(1), 341-366, 1952.

\bibitem {fa1} M. H. Faroughi,  and R. Ahmadi, 
\textit{Some Properties of C-Frames of Subspaces.}
J. Nonlinear Sci. Appl. \textbf{1}(3), 155-168, 2008.

\bibitem {fe} H. G.  Feichtinger,  and T. Werther, 
\textit{Atomic Systems for Subspaces.}
Proceedings SampTA. Orlando, FL, 163-165, 2001.

\bibitem {ga} L. Gavruta, 
\textit{Frames for Operators.}
Appl. Comp. Harm. Annal. \textbf{32}, 139-144, 2012.

\bibitem {hm} S. B.  Heineken,  P. M. Morillas, A. M. Benavente,  and M. I. Zakowich, 
\textit{Dual Fusion Frames.}
arXiv: 1308.4595v2 [math.CA] 29 Apr 2014.

\bibitem {he}H.  Heuser, 
\textit{Functional Analysis.}
John Wiley, New York, 1991.

\bibitem {kh}  M. Khayyami, and A. Nazari, 
\textit{Construction of Continuous g-Frames and Continuous Fusion Frames.}
SCMA \textbf{4}(1), 43-55, 2016.

\bibitem {naf} A. Najati,  M. H. Faroughi,  and  A. Rahimi,  
{\em g-frames and stability of g-frames in Hilbert spaces.}
  Methods of Functional Analysis and Topology.  \textbf{14}(3),  305--324, 2008.
  
\bibitem {sad} V. Sadri and R. Ahmadi,  
{\em Some Results on Frames and g-Frames.}
Submited 2017.

\bibitem {sun} W. Sun, 
 {\em G-Frames and G-Riesz bases\/},
   J. Math. Anal. Appl.326 , 437-452, 2006.
\end{thebibliography}
\end{document}